\documentclass[11pt]{article}
\usepackage[utf8]{inputenc}
\usepackage{amsmath, amssymb, amsfonts, tikz, hyperref, comment, url, todonotes, wasysym, import, geometry, amsthm, mathtools, subfig}
\usepackage[capitalise]{cleveref}
\usepackage[export]{adjustbox}
\usepackage{setspace}
\usepackage{indentfirst}
\newtheorem{theorem}{Theorem}[section]
\newtheorem{lemma}[theorem]{Lemma}

\newtheorem{conjecture}[theorem]{Conjecture}

\theoremstyle{definition}
\newtheorem{example}{Example}[section]
\newtheorem{definition}{Definition}[section]

\theoremstyle{remark}
\newtheorem{remark}{Remark}[section]

\makeatletter
\renewcommand*\env@matrix[1][\arraystretch]{
  \edef\arraystretch{#1}
  \hskip -\arraycolsep
  \let\@ifnextchar\new@ifnextchar
  \array{*\c@MaxMatrixCols c}}
\makeatother

\begin{document}

\title{The Rank-Generating Functions of Upho Posets}
\author{Yibo Gao, Joshua Guo, Karthik Seetharaman, Ilaria Seidel}
\maketitle

\begin{center}
\begin{abstract}
    Upper homogeneous finite type (upho) posets are a large class of partially ordered sets with the property that the principal order filter at every vertex is isomorphic to the whole poset. Well-known examples include $k$-array trees, the grid graphs, and the Stern poset. Very little is known about upho posets in general. In this paper, we construct upho posets with Schur-positive Ehrenborg quasisymmetric functions, whose rank-generating functions have rational poles and zeros. We also categorize the rank-generating functions of all planar upho posets. Finally, we prove the existence of an upho poset with uncomputable rank-generating function.
\end{abstract} 
\end{center}\noindent 

\section{Introduction}

In this paper, we study a particularly natural symmetry condition on posets called \textit{upper homogeneous finite}, or \textit{upho}. This condition says that $P$ is isomorphic to every principal order filter of itself, which is defined as follows. 

\begin{definition}[\cite{Stanley-Sterns}]
A poset $P$ is \textit{upho} if it is of finite type and the principal order filter $V_{P,s} \cong P$ for all $s \in P$.\end{definition}

We will provide the necessary background to this definition in Section~\ref{sec: preliminaries}.

Upho posets, recently introducd by Richard Stanley \cite{Stanley-Sterns}, contain large families of infinite posets, including the square grid, $k$-array trees, Stern's poset and so on. Stanley focused on a particular upho poset known as Stern's poset. Named after its similarity to Stern's triangle, it gives rise to many interesting enumeration problems (see \cite{speyer2019proof} and \cite{yang2020stanley}). We provide more detailed examples \ref{example: chain}, \ref{example: 2-tree}, and \ref{example: cartesian coordinates} in Section~\ref{sec: preliminaries}.

One particularly interesting example of an upho poset arises in commutative algebra. If $R$ is a local principal ideal domain with a finite residue field, then the poset given by submodules of $R^d$, $d\in\mathbb{Z}_{>0}$, of finite colength ordered by reverse inclusion is a modular lattice and an upho poset. Here, the lattice operations, meet and join, are given by sum and intersection of modules. In fact, the converse is conjectured to be almost true. We record the following conjecture by Stanley in the literature, but we will not discuss it further in this paper.
\begin{conjecture}\label{conj:indecomposable-modular}
Every indecomposable upho modular lattice that contains a rank 3 interval that is complemented, is isomorphic to the poset given by submodules of $R^d$ of finite colength ordered by reverse inclusion for some local principal ideal domain $R$ with a finite residue field.
\end{conjecture}

We study three types of upho posets in this paper. Firstly, because the Ehrenborg quasisymmetric functions associated with upho posets turn out to be symmetric (Lemma~\ref{lemma:ehrenborg is product of rgfs}), it is natural to ask about the existence of upho posets with Schur-positive Ehrenborg quasisymmetric functions. Thus, in Section~\ref{sec: schur-positivity}, we construct a large class of upho posets with Schur-positive Ehrenborg quasisymmetric function, particularly those with rational poles and zeros, in order to prove the following theorem.

\begin{theorem}\label{schur-positive 1-bx}
Given positive integers $a_1,a_2,\ldots,a_n$ and $b_1,b_2,\ldots,b_m$ with $n \geq 0$ and $m \geq 1$, there exists an upho poset $P$ with rank-generating function $$\frac{(1+a_1x)(1+a_2x)\cdots(1+a_nx)}{(1-b_1x)(1-b_2x)\cdots(1-b_mx)}.$$
\end{theorem}

Then, in Section~\ref{sec: planar}, we precisely categorize the rank-generating functions of all upho posets with a planar Hasse diagram. 

\begin{theorem}
The rank-generating function of any planar upho poset $P$ with up-degree $b$ is of the form $Q(x)^{-1}$ where $Q(x)=1-bx+a_2x^2+a_3x^3+ \cdots +a_nx^n$ such that $b, a_1, a_2, \ldots, a_n \in \mathbb{Z}_{\geq{0}}$ and $Q(1)\leq 0$. Furthermore, any such function $Q^{-1}(X)$ is realized by some planar upho poset.
\end{theorem}

Finally, after focusing on upho posets with rational rank-generating function, we prove the existence of an upho poset with uncomputable rank-generating function in Section~\ref{sec: irrationality}. To do so, we define a new class of upho posets constructed from monoids with a finite alphabet, homogeneous relations and certain restrictions. This allows us to simplify our upho conditions and more easily consider the rank-generating functions. 

\begin{theorem} \label{thm:intro-irrationality}
There exists an upho poset with uncomputable rank-generating function.
\end{theorem}
\section{Preliminaries}\label{sec: preliminaries}

We refer the readers to \cite{ec1} for a detailed exposition. In a poset $P$, we say that some element $y \in P$ \textit{covers} $x \in P$, denoted $x \lessdot y$ if $x < y$ and there does not exist $z \in P$ with $x < z < y$. We can then define the \textit{up-degree} of $x$ to be the number of elements that cover $x$, and the \textit{down-degree} of $x$ to be the number of elements covered by $x$. Additionally, let $\hat{0}$ be the unique minimal element of $P$ if such an element exists, and we call $x \in P$ an \textit{atom} if $\hat{0} \lessdot x$.

We call $P$ \textit{ranked} if there exists a disjoint union partition of $P$, $P_0 \sqcup P_1 \sqcup P_2 \sqcup \ldots$, called the rank decomposition of $P$, with the property that, if $x \in P_i$ and $x \lessdot y$, then $y \in P_{i+1}$. We say $P_i$ is the $i$-th rank of $P$. If $x \in P_i$, then we write $\rho(x) = i$. If $P$ is \textit{ranked}, we define the \textit{rank-generating function} of $P$ to be $$F_P(x)=\sum_{k=0}^{\infty}|P_k|x^k.$$ Additionally, a ranked poset $P$ is of \textit{finite type} if $|P_i|$ is finite for all nonnegative integers $i$. 

Recall that an \textit{isomorphism} on two posets $P$ and $Q$ is a bijection $\iota: P \rightarrow Q$ such that, for elements $p_1, p_2 \in P$ and $q_1, q_2 \in Q$, if $p_1 > p_2$ then $\iota(p_1) > \iota(p_2)$ and if $q_1 > q_2$ then $\iota^{-1}(q_1) > \iota^{-1}(q_2)$. We also define $V_{P,s} :=\{t\in P\:|\:t\geq s\}$, the \textit{principal order filter} above $s \in P$. When the poset referred to is unambiguous, we may abbreviate $V_{P,s}$ to $V_s$. In Figure~\ref{fig: preliminaries}, the subposets shown in dashed lines are each an principal order filter above the highlighted vertex. 

The following definition is due to Richard Stanley \cite{Stanley-Sterns}, although it is not found in the literature.
\begin{definition}[\cite{Stanley-Sterns}]\label{def: upho}
A ranked poset $P$ is \textit{upper homogeneous finite} (or \textit{upho}) if it is of finite type and has the property that the principal order filter $V_{P,s} \cong P$ for all $s \in P$.

\end{definition}
We provide some examples of upho posets. 

\begin{example}\label{example: chain}
Consider the poset defined by $\mathbb{N}$ with the standard ordering $<$. Then for any $s \in \mathbb{N}$, $\iota: V_s \rightarrow \mathbb{N}$ takes an element $x>s$ to $x-s$. Note that this is bijective and preserves ordering: if $x>y>s$, then $x-s>y-s$, and if $x>y$, then $x+s>y+s$, as desired.
\end{example}

\begin{example}\label{example: 2-tree}
A full binary tree $P$ is a simple example of an upho poset. For any $s$ on rank $i$ of $P$, we describe the isomorphism $\iota: V_s \rightarrow P$ as follows. Let the Hasse Diagram of the poset be as shown in Figure~\ref{fig: preliminaries}. Assign each element on rank $k$ a $k$-element binary string, such that the $k$-th element from the left is represented by $k-1$ in binary. Then, we can define $\iota: V_s \rightarrow P$ by taking every element of $V_s$ and deleting the prefix $s$. 
\end{example}

\begin{example}\label{example: cartesian coordinates}
Consider a poset $P = \mathbb{N} \times \mathbb{N}$ consisting of all $2$-dimensional Cartesian lattice points with nonnegative coordinates (see Figure~\ref{fig: preliminaries}). Then let $(x_1, y_1) \geq (x_2, y_2)$ if and only if $x_1 \geq x_2$, $y_1 \geq y_2$. For any point $s = (s_x, s_y)$, we have an isomorphism $\iota: V_s \rightarrow P$ taking any element $(x , y) \geq (s_x, s_y)$ to $(x - s_x, y - s_y)$. Thus, the poset is upho.
\end{example}

The Stern Poset, introduced by Stanley in \cite{Stanley-Sterns}, as well as the nicknamed ``Bowtie" poset, are also examples of upho posets. Diagrams of these two posets are also included in Figure~\ref{fig: preliminaries}.

\begin{figure}
 \centering
 \includegraphics[width = 450pt]{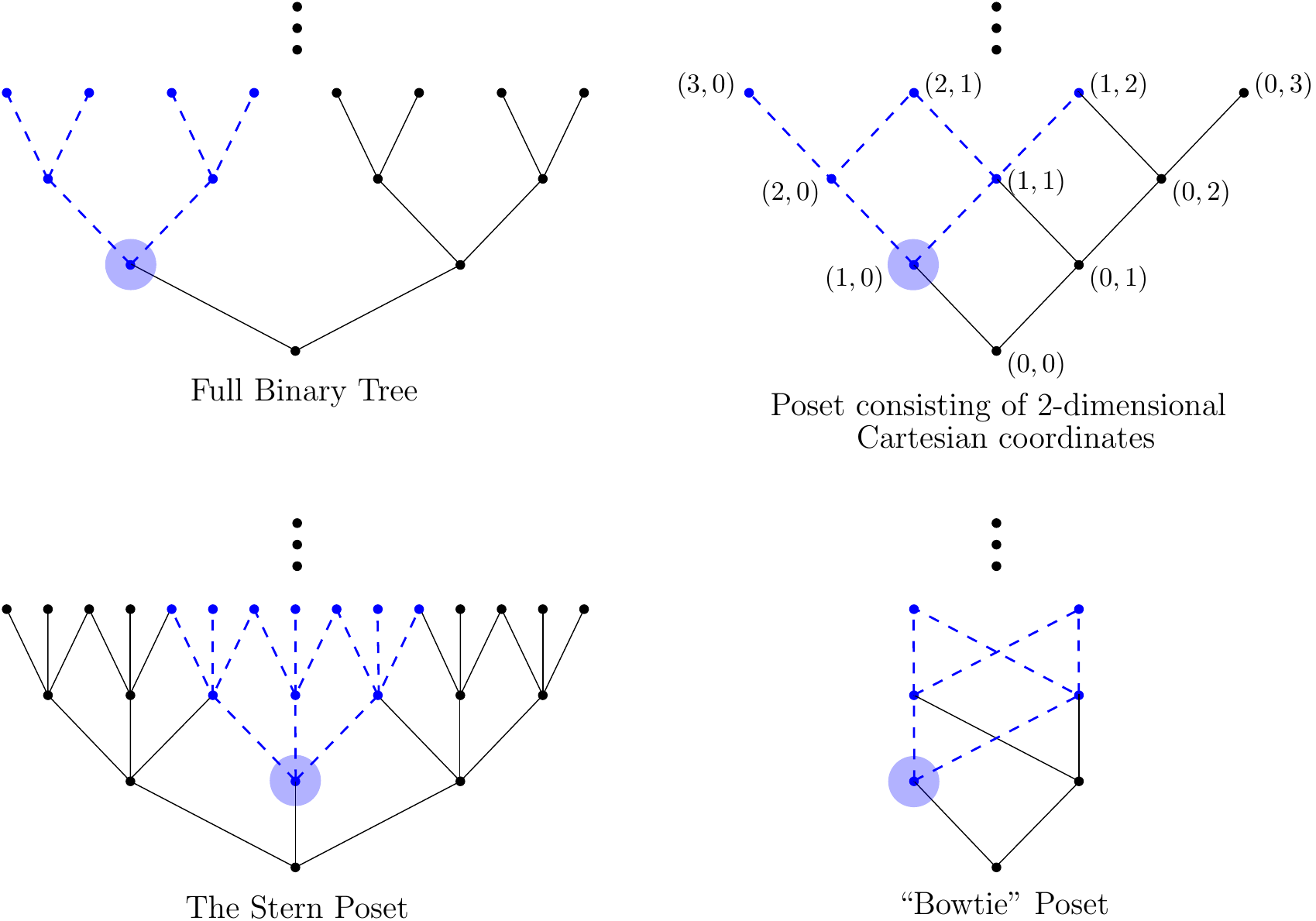}
 \caption{Examples of upho posets. Cover relations in the poset consisting of $2$-dimensional Cartesian coordinates are defined by $(x_1,y_1) \geq (x_2,y_2)$ if and only if the points are distinct and $x_1 \geq x_2$, $y_1 \geq y_2$.}
 \label{fig: preliminaries}
\end{figure}

We establish some useful lemmas regarding upho posets.

\begin{lemma}\label{lem: upho-posets-infinite}
A nonempty upho poset $P$ either consists of a single element or infinitely many elements. Furthermore, any upho poset has a unique minimal element $\hat{0}$. 
\end{lemma}

\begin{proof}

If $P$ is nonempty, then it has a minimal element $a$. Since $V_a \cong P$ and $V_a$ has exactly one minimal element, $P$ must also have exactly one minimal element. If $P$ has at least two elements, then it contains at least one non-minimal element $b$. Then $V_b$ is a strict subset of $P$ in bijection with $P$, so $P$ is infinite.
\end{proof}

We focus on infinite upho posets for the remainder of this paper. We define and provide lemmas regarding the following function associated with posets.

\begin{definition}[\cite{ehrenborg1996posets}]\label{def:ehrenborg-quasisym}
For any ranked poset $P$ of finite type with a minimum $\hat0$, define its \textit{Ehrenborg quasisymmetric function} of degree $n$ to be 
$$E_{P,n}(x_1,x_2,\ldots):=\sum_{\substack{\hat0=t_0\leq t_1\leq\cdots\leq t_{k-1}<t_k\\\rho(t_k)=n}}x_1^{\rho(t_0,t_1)}x_2^{\rho(t_1,t_2)}\cdots x_k^{\rho(t_{k-1},t_k)}$$
where $\rho(t_i,t_{i+1})=\rho(t_{i+1})-\rho(t_i)$. We also write $E_P:=\sum_{n\geq0} E_{P,n}$.
\end{definition}

First, we have the following lemma on the multiplicativity of upho posets and their Ehrenborg quasisymmetric functions.

\begin{lemma}\label{lem:multiplicative}
Let $P$ and $Q$ be upho posets. Then, $P \times Q$ is an upho poset. Furthermore, $F_{P \times Q} = F_P F_Q$ and $E_{P \times Q} = E_PE_Q$. 
\end{lemma}

\begin{proof}
Since $P$ and $Q$ are of finite type, $P \times Q$ is obviously of finite type as well. Now, take elements $p \in P$ and $q \in Q$. Then, $V_{P \times Q, (p,q)} \equiv V_{P,p} \times V_{Q,q} \equiv P \times Q$, implying $P \times Q$ is upho, as desired. 

To show $F_{P \times Q} = F_PF_Q$, note that $(P\times Q)_k = \cup_{t = 0}^k (P_t\times Q_{k-t})$, so that $|(P\times Q)_k| = \sum_{t = 0}^k |P_t||Q_{k-t}|$. This implies that
\begin{align*}
    F_{P\times Q} &= \sum_{k=0}^{\infty} \left(\sum_{t=0}^k |P_t||Q_{k-t}|\right)x^k \\
    &= \left(\sum_{k_1=0}^{\infty} |P_{k_1}|x^{k_1}\right)\left(\sum_{k_2=0}^{\infty} |Q_{k_1}|x^{k_1}\right) \\
    &= F_P\cdot F_Q
\end{align*}
as desired. 

By lemma \ref{lemma:ehrenborg is product of rgfs}, showing $F_{P \times Q} = F_PF_Q$ suffices to show $E_{P \times Q} = E_PE_Q$, and we are done.
\end{proof}

\begin{example}
Let $P$ and $Q$ contain the elements of $\mathbb{N}$ with the standard binary operation $>$. As shown in Example~\ref{example: chain}, both $P$ and $Q$ are upho. Then $P \times Q$ yields an upho poset consisting of $2$-dimensional Cartesian coordinates, described in Example~\ref{example: cartesian coordinates}.
\end{example}

We also have the following important lemma relating $E_P$ and $F_P$. 

\begin{lemma}\label{lemma:ehrenborg is product of rgfs}
Let $P$ be upho. Then $E_P$ is a symmetric function. Moreover, $$E_P(x_1,x_2,\ldots)=F_P(x_1)F_P(x_2)\cdots.$$
\end{lemma}
\begin{proof}

It suffices to prove that $E_{P,n}(x_1,x_2,\ldots)$ is the coefficient of $x^n$ in $\prod_{i=1}^{\infty} F_P(x_i)$. First, note that the coefficient of $x^n$ in $\prod_{i=1}^{\infty} F_P(x_i)$ is $\sum_{S} \prod{|r_{\alpha_i}|}$, where $S$ is the set of all sequences $(\alpha_1,\alpha_2,\ldots)$ with $\sum_{i=1}^{\infty} \alpha_i = n$. (Here, $|r_{\alpha_i}|$ is the cardinality of rank $\alpha_i$ in $P$). 

Take a chain $\hat{0}=t_0 \leq t_1 \leq \cdots \leq t_{k-1} < t_k, \rho(t_k)=n$. Then, let $\alpha_i = \rho(t_i, t_{i-1})$ for all $1 \leq i \leq k$ and $\alpha_i=0$ for $i > k$. We can do the same in reverse, which yields a bijection between all sequences $\{\alpha_i\}$ that have terms summing to $n$ and chains of $P$ with maximal element on rank $n$ and first element $\hat{0}$. Thus, every chain of $P$ with elements $\hat{0} = t_0, t_1,\ldots, t_k$ with $\rho(t_k)=n$ can be described as $\rho(t_i) = \sum_{j=1}^i \alpha_i$. 

It now suffices to count the number of such chains. There are $|r_{\alpha_1}|$ choices for $t_1$. Say we pick $t_1 = k \in P$, where $\rho(k) = \alpha_1$. Then, $t_2$ must equal $l$ where $l \geq k$ and $\rho(l) = \alpha_1+\alpha_2$. There are $|r_{\alpha_2}|$ such choices in $V_{k}$, for $|r_{\alpha_1}||r_{\alpha_2}|$ choices, and continuing yields there are $\sum_{S} \prod{|r_{\alpha_i}|}$ chains of the desired form, which implies the conclusion by coefficient matching.
\end{proof}

Our main focus of Section~\ref{sec: schur-positivity} is on upho posets whose Ehrenborg quasisymmetric functions, which is a symmetric function in this case, is Schur-positive. To this end, the following theorem by Davydov \cite{davydov2000totally},
which has great importance in the theory of total positivity, is highly useful.

\begin{theorem}\label{thm:davydov}
An integral series $f(t)\in 1+t\mathbb{Z}[[t]]$ is totally positive, i.e. $f(t_1)f(t_2)\cdots$ is Schur positive, if and only if it is of the form $f(t) = g(t)/h(t)$ where $g(t),h(t)\in\mathbb{Z}[t]$ such that all the complex roots of $g(t)$ are negative real numbers and all the complex roots of $h(t)$ are positive real numbers.
\end{theorem}

In this paper, we construct a large family of Schur-positive upho posets, specifically those whose rank-generating functions have rational zeros and poles: $$F_P(x) = \frac{\prod_{i=1}^n (1+a_ix)}{\prod_{i=1}^m (1-b_ix)}.$$

Then, in Section~\ref{sec: planar}, we aim to categorize the rank-generating functions of all planar posets. We define a planar poset as follows.

\begin{definition}
A ranked poset $P$ is \textit{planar} if there exists a Hasse diagram of $P$ such that every vertex on rank $i$ of $P$ is at $y$-coordinate $i$ and no two edges of the Hasse diagram intersect.
\end{definition}

\begin{example}
As shown in Figure~\ref{fig: preliminaries}, the full binary tree is planar: in the Hasse diagram of the poset, we represent every vertex on rank $i$ at $y$-coordinate $i$. Furthermore, no two edges in the diagram intersect, as desired.
\end{example}
\begin{example}
In Figure~\ref{fig: preliminaries}, the ``Bowtie" poset is not planar, as there is no way to represent the poset with the vertices on rank $i$ at $y$-coordinate $i$ without the edges intersecting. This can be rigorously proven using Lemma~\ref{lem:meet semilattice}, because no two vertices on the same rank of the poset have a unique meet.
\end{example}

Finally, in Section~\ref{sec: irrationality}, we prove the existence of an upho poset with uncomputable rank-generating function. 

\section{Schur-positive upho posets}\label{sec: schur-positivity}

In this section, we prove the existence of a large family of upho posets with Schur-positive Ehrenborg quasisymmetric rank-generating functions. We specifically focus on those with rational zeros and poles. We begin with the following lemma.

\begin{lemma}\label{schur-positive 1-x}
Given positive integers $a_1,a_2,\ldots,a_n$, there exists an upho poset $P$ with rank-generating function $$F_P(x) = \frac{(1+a_1x)(1+a_2x)\cdots(1+a_nx)}{1-x}.$$
\end{lemma}

\begin{proof}
Let $S$ be the set of all points $(y_1,y_2,\ldots,y_n)$ with $0 \leq y_i \leq a_i$ for all $1 \leq i \leq n$. We define a poset $P$ on $S$, with each element of $P$ belonging to $S$. Let $(y_1,y_2,\ldots,y_n;k)$ denote the element of $P$ on rank $k$ represented by $(y_1,\ldots,y_n)$. Then, define $$(y_1,\ldots,y_n;k) \lessdot (z_1,\ldots,z_n;k+1)$$ if and only if $z_i \neq y_i$ for at most one $1 \leq i \leq n$. See Figure~\ref{fig:gridconstruction} for an example with $a_1=1$, $a_2=2$. We begin by proving that $P$ has the desired rank-generating function.

In what follows, let $e_i$ be the $i$-th elementary symmetric polynomial in variables $a_1,a_2,\ldots,a_n$ for $1 \leq i \leq n$. Furthermore, let $e_0 = 1$ and $e_j = 0$ for $j > n$. Note that $$\prod_{i=1}^n (1+a_ix) = \sum_{i=0}^{\infty} e_ix^i.$$ 

Now, if $r_i$ is the number of elements on rank $i$ of $P$, it suffices to show that $(1-x)\sum_{i=0}^{\infty} r_ix^i = \sum_{i=0}^{\infty} e_ix^i$. 

It suffices to show that the coefficient of each $x^i$ term is the same on both sides of the equation. It is clear that $r_0=e_0=1$, so we are left to show that $r_i -r_{i-1} = e_i$ for all $i > 0$. Note that rank $i$ of $P$ consists of all points in $S$ with at most $i$ nonzero coordinates. There are $e_j$ points in $S$ with exactly $j$ nonzero coordinates, so $r_i = \sum_{j=0}^i e_j$. This easily gives $r_i - r_{i-1} = e_i$, as desired.

We have proved that our construction satisfies the desired rank-generating function, so it remains to show that our poset is upho.

Clearly, $P$ is of finite type. It then suffices to provide an isomorphism $\iota_p:P \longrightarrow V_p$ for any $p \in P$. Let $(y_1,\ldots,y_n;k)$ denote the element of $P$ on rank $k$ that is represented by the point $(y_1,\ldots,y_n)$. Let $p = (p_1,p_2,\ldots,p_n) \in P$. Then, define $$\iota_p: (y_1, \ldots, y_n; k) \mapsto (y_1 + p_1, \ldots, y_n + p_n; k + k_p)$$ where $y_i \in \mathbb{Z}/(a_i+1)\mathbb{Z}$ for all $i$. 
That is, we add $p_i$ to the $i$-th coordinate for all $i$ and take it modulo $a_i+1$.

We have $\iota_p(y_1, \ldots, y_n; k) \geq p$, so that $\iota_p$ maps every element of $P$ to a unique element of $V_p$. Finally, note that $(y_1, \ldots, y_n)$ and $(y_1',\ldots, y_n')$ differ in exactly one coordinate iff $(y_1 + p_1, \ldots, y_n + p_n)$ and $(y_1' + p_1, \ldots, y_n' + p_n)$ also differ in exactly one coordinate.
This implies that $$(y_1, \ldots, y_n; k)\lessdot (y_1',\ldots, y_n'; k + 1) \Longleftrightarrow (y_1 + p_1, \ldots, y_n + p_n; k + k_p)\lessdot (y_1' + p_1, \ldots, y_n' + p_n; k + k_p + 1),$$
so that $\iota_p$ preserves the covering relations of $P$. Thus, $\iota_p$ is an isomorphism between $P$ and $V_p$ for all $p\in P$, so $P$ is upho.
\end{proof}

\begin{figure}[h!]
\centering
\begin{tikzpicture}[scale = 0.300]
\draw(0,0)--(3,0)--(3,2)--(0,2)--(0,0);
\draw(0,1)--(3,1);
\draw(1,0)--(1,2);
\draw(2,0)--(2,2);
\draw(0,1)--(1,2);
\draw(0,2)--(1,1);
\draw(6,0)--(9,0)--(9,2)--(6,2)--(6,0);
\draw(6,1)--(9,1);
\draw(7,0)--(7,2);
\draw(8,0)--(8,2);
\draw(7,1)--(8,2);
\draw(7,2)--(8,1);
\draw(12,0)--(15,0)--(15,2)--(12,2)--(12,0);
\draw(12,1)--(15,1);
\draw(13,0)--(13,2);
\draw(14,0)--(14,2);
\draw(14,1)--(15,2);
\draw(14,2)--(15,1);
\draw(18,0)--(21,0)--(21,2)--(18,2)--(18,0);
\draw(18,1)--(21,1);
\draw(19,0)--(19,2);
\draw(20,0)--(20,2);
\draw(18,0)--(19,1);
\draw(18,1)--(19,0);
\draw(24,0)--(27,0)--(27,2)--(24,2)--(24,0);
\draw(24,1)--(27,1);
\draw(25,0)--(25,2);
\draw(26,0)--(26,2);
\draw(25,0)--(26,1);
\draw(25,1)--(26,0);
\draw(30,0)--(33,0)--(33,2)--(30,2)--(30,0);
\draw(30,1)--(33,1);
\draw(31,0)--(31,2);
\draw(32,0)--(32,2);
\draw(32,0)--(33,1);
\draw(32,1)--(33,0);
\draw(0,8)--(3,8)--(3,10)--(0,10)--(0,8);
\draw(0,9)--(3,9);
\draw(1,8)--(1,10);
\draw(2,8)--(2,10);
\draw(0,9)--(1,10);
\draw(0,10)--(1,9);
\draw(6,8)--(9,8)--(9,10)--(6,10)--(6,8);
\draw(6,9)--(9,9);
\draw(7,8)--(7,10);
\draw(8,8)--(8,10);
\draw(7,9)--(8,10);
\draw(7,10)--(8,9);
\draw(12,8)--(15,8)--(15,10)--(12,10)--(12,8);
\draw(12,9)--(15,9);
\draw(13,8)--(13,10);
\draw(14,8)--(14,10);
\draw(14,9)--(15,10);
\draw(14,10)--(15,9);
\draw(18,8)--(21,8)--(21,10)--(18,10)--(18,8);
\draw(18,9)--(21,9);
\draw(19,8)--(19,10);
\draw(20,8)--(20,10);
\draw(18,8)--(19,9);
\draw(18,9)--(19,8);
\draw(24,8)--(27,8)--(27,10)--(24,10)--(24,8);
\draw(24,9)--(27,9);
\draw(25,8)--(25,10);
\draw(26,8)--(26,10);
\draw(25,8)--(26,9);
\draw(25,9)--(26,8);
\draw(30,8)--(33,8)--(33,10)--(30,10)--(30,8);
\draw(30,9)--(33,9);
\draw(31,8)--(31,10);
\draw(32,8)--(32,10);
\draw(32,8)--(33,9);
\draw(32,9)--(33,8);
\draw(6,-8)--(9,-8)--(9,-6)--(6,-6)--(6,-8);
\draw(6,-7)--(9,-7);
\draw(7,-8)--(7,-6);
\draw(8,-8)--(8,-6);
\draw(6,-7)--(7,-6);
\draw(6,-6)--(7,-7);
\draw(12,-8)--(15,-8)--(15,-6)--(12,-6)--(12,-8);
\draw(12,-7)--(15,-7);
\draw(13,-8)--(13,-6);
\draw(14,-8)--(14,-6);
\draw(13,-7)--(14,-6);
\draw(13,-6)--(14,-7);
\draw(18,-8)--(21,-8)--(21,-6)--(18,-6)--(18,-8);
\draw(18,-7)--(21,-7);
\draw(19,-8)--(19,-6);
\draw(20,-8)--(20,-6);
\draw(20,-7)--(21,-6);
\draw(20,-6)--(21,-7);
\draw(24,-8)--(27,-8)--(27,-6)--(24,-6)--(24,-8);
\draw(24,-7)--(27,-7);
\draw(25,-8)--(25,-6);
\draw(26,-8)--(26,-6);
\draw(24,-8)--(25,-7);
\draw(24,-7)--(25,-8);
\draw(15,-16)--(18,-16)--(18,-14)--(15,-14)--(15,-16);
\draw(15,-15)--(18,-15);
\draw(16,-16)--(16,-14);
\draw(17,-16)--(17,-14);
\draw(15,-15)--(16,-14);
\draw(15,-14)--(16,-15);
\draw(1.50,2.30)--(1.50,7.70);
\draw(1.50,2.30)--(7.50,7.70);
\draw(1.50,2.30)--(13.5,7.70);
\draw(1.50,2.30)--(19.5,7.70);
\draw(7.50,2.30)--(1.50,7.70);
\draw(7.50,2.30)--(7.50,7.70);
\draw(7.50,2.30)--(13.5,7.70);
\draw(7.50,2.30)--(25.5,7.70);
\draw(13.5,2.30)--(1.50,7.70);
\draw(13.5,2.30)--(7.50,7.70);
\draw(13.5,2.30)--(13.5,7.70);
\draw(13.5,2.30)--(31.5,7.70);
\draw(19.5,2.30)--(1.50,7.70);
\draw(19.5,2.30)--(19.5,7.70);
\draw(19.5,2.30)--(25.5,7.70);
\draw(19.5,2.30)--(31.5,7.70);
\draw(25.5,2.30)--(7.50,7.70);
\draw(25.5,2.30)--(19.5,7.70);
\draw(25.5,2.30)--(25.5,7.70);
\draw(25.5,2.30)--(31.5,7.70);
\draw(31.5,2.30)--(13.5,7.70);
\draw(31.5,2.30)--(19.5,7.70);
\draw(31.5,2.30)--(25.5,7.70);
\draw(31.5,2.30)--(31.5,7.70);
\draw(7.50,-5.70)--(1.50,-0.300);
\draw(7.50,-5.70)--(7.50,-0.300);
\draw(7.50,-5.70)--(13.5,-0.300);
\draw(7.50,-5.70)--(19.5,-0.300);
\draw(13.5,-5.70)--(1.50,-0.300);
\draw(13.5,-5.70)--(7.50,-0.300);
\draw(13.5,-5.70)--(13.5,-0.300);
\draw(13.5,-5.70)--(25.5,-0.300);
\draw(19.5,-5.70)--(1.50,-0.300);
\draw(19.5,-5.70)--(7.50,-0.300);
\draw(19.5,-5.70)--(13.5,-0.300);
\draw(19.5,-5.70)--(31.5,-0.300);
\draw(25.5,-5.70)--(1.50,-0.300);
\draw(25.5,-5.70)--(19.5,-0.300);
\draw(25.5,-5.70)--(25.5,-0.300);
\draw(25.5,-5.70)--(31.5,-0.300);
\draw(16.5,-13.7)--(7.50,-8.30);
\draw(16.5,-13.7)--(13.5,-8.30);
\draw(16.5,-13.7)--(19.5,-8.30);
\draw(16.5,-13.7)--(25.5,-8.30);
\node at (16.5,13) {$\vdots$};
\node at (-4,-15) {$P_0$};
\node at (-4,-7) {$P_1$};
\node at (-4,1) {$P_2$};
\node at (-4,9) {$P_3$};
\end{tikzpicture}
\caption{The construction of Lemma~\ref{schur-positive 1-x} with $a_1 = 1$, $a_2 = 2$.}
\label{fig:gridconstruction}
\end{figure}
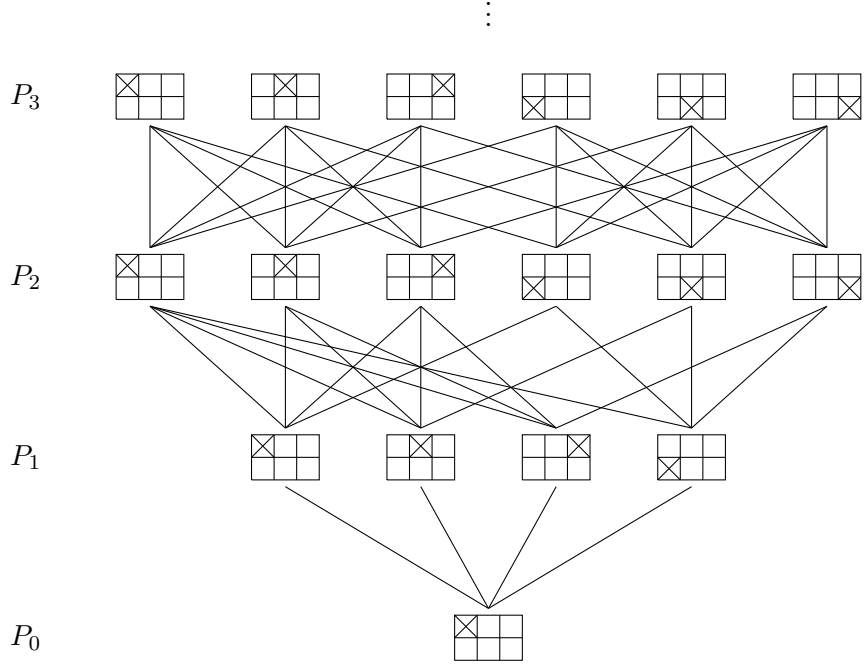

Using this construction, we can prove the following:

\begin{lemma}\label{schur-positive 1-bx mini}
Given positive integers $a_1,a_2,\ldots,a_n,b$, there exists an upho poset $Q$ with rank-generating function $$F_Q(x) = \frac{(1+a_1x)(1+a_2x)\cdots(1+a_nx)}{1-bx}.$$
\end{lemma}

\begin{proof}
We begin by describing the construction of an upho poset $Q$ with the desired rank-generating function. We then prove that $Q$ has the desired rank-generating function and is upho.

We make use of Lemma~\ref{schur-positive 1-x}, which constructs an upho poset $P$ with rank-generating function $$\frac{(1+a_1x)(1+a_2x)\cdots(1+a_nx)}{1-x}.$$ We let $P_i$ be the $i$th rank of $P$ with $|P_i|=s_i$ and $Q_i$ be the $i$th rank of $Q$ with $|Q_i|=r_i$. 

We construct $Q$ recursively, starting from $\hat{0}$ and adding one rank at a time. Suppose we have constructed up to rank $i-1$ of the poset. We describe rank $i$ as a union of $i+1$ disjoint groups of vertices $G_{i,j}$ for $0 \leq j \leq i$, where $j \in \mathbb{Z}$. Let $G_{i,0}$ be a copy of the $i$th rank of $P$, so that $|G_{i, 0}|=  s_i$, and let $G_{i,j}$ consist of $b-1$ sets of vertices that are copies of $Q_{j-1}$, so that $|G_{i, j}| = (b-1)r_{j-1}~\forall~ 1 \leq j \leq i$. We now construct cover relations between ranks $i-1$ and $i$. Firstly, cover relations between $G_{i-1,0}$ and $G_{i,0}$ are isomorphic to those between $P_{i-1}$ and $P_i$ in $P$. We also let every element of $G_{i,1}$ cover every element of $G_{i-1,0}$. 

To connect $G_{i-1,j}$ to rank $i$ for $0<j \leq i-1$, note that $G_{i-1,j}$ consists of $b-1$ sets of vertices that are copies of $Q_{j-1}$, and $G_{i,j+1}$ consists of $b-1$ copies of $Q_j$. Thus, between $G_{i-1,j}$ and $G_{i,j+1}$, we add cover relations isomorphic to those between $Q_{j-1}$ and $Q_j~\forall~ 1 \leq j \leq i-1$. 

\begin{figure}
    \centering
    \includegraphics[width = 440pt]{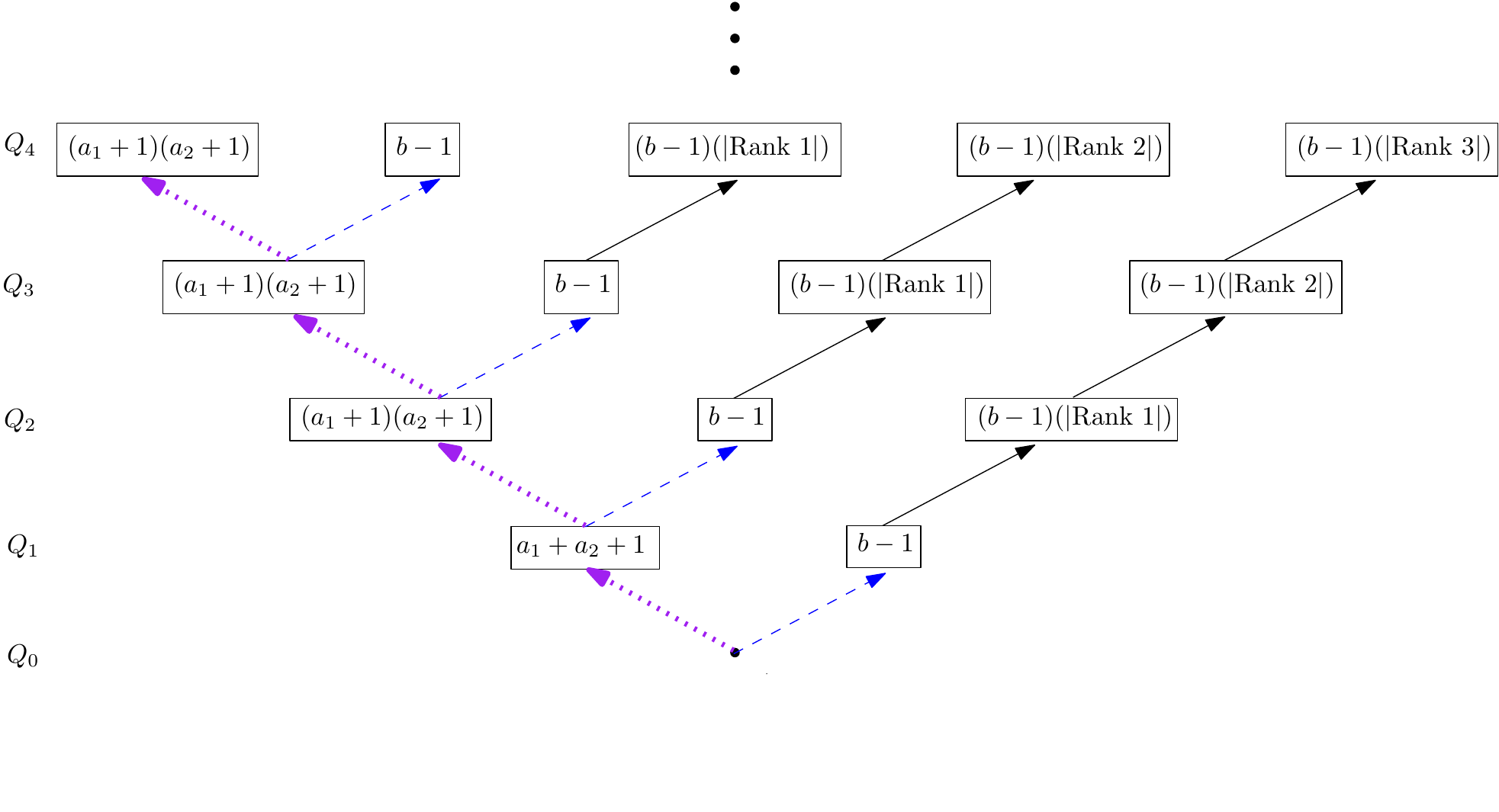}
    \caption{This figure shows the upho poset $Q$ with rank-generating function $\frac{(1+a_1x)(1+a_2x)}{1-bx}$. Labelled boxes indicate groups of vertices, and $|\text{Rank }n|$ represents the number of vertices on rank $n$ of $Q$.}
    \label{1-bx figure}
\end{figure}

Figure~\ref{1-bx figure} shows the construction for rank-generating function $\frac{(1+a_1x)(1+a_2x)}{1-bx}$. Dotted arrows indicate sets of cover relations isomorphic to those in the poset described in Lemma~\ref{schur-positive 1-x}. Dashed arrows denote complete bipartite graphs, where every element in the higher group covers every element in the lower group. Note that each group of $b-1$ elements starts $b-1$ copies of $Q$, so solid arrows denote $b-1$ sets of cover relations isomorphic to those found between earlier ranks in the poset. For example, the solid arrow between $G_{1,1}$ and $G_{2,2}$ denotes $b-1$ copies of the cover relations between $Q_0$ and $Q_1$. 

Now that we have described our construction, we must prove that it has the desired rank-generating function. Let $F_Q(x) = \sum_{i=0}^{\infty} r_ix^i$, where $r_i$ is the number of elements in rank $i$. Also, let $e_i$ be the $i$-th elementary symmetric polynomial in variables $a_1,\ldots,a_n$, with $e_0=1$ and $e_j=0$ for all $j > n$.

Then, for the rank-generating function of $Q$ to be $$F_Q(x) = \frac{(1+a_1x)(1+a_2x)\cdots(1+a_nx)}{1-bx},$$ we must have $r_n-br_{n-1}=e_n$ by coefficient matching. Note that $$r_n = \sum_{i=0}^n |G_{n,i}| = s_n + (b-1)(r_0+\cdots+r_{n-1}) = (b-1)(r_0+\cdots+r_{n-1}) + \sum_{i=0}^n e_i,$$ which means that $$r_{n+1} = (b-1)(r_0+\cdots+r_n) + \sum_{i=0}^{n+1}e_i.$$ This implies $$r_{n+1} - r_n = (b-1)r_n + e_{n+1} \implies r_{n+1} - br_n = e_{n+1}.$$ Together with the fact that $r_0=1$, this implies that we have the desired rank-generating function.

Next, we must prove that the poset is upho.

Consider an element $q \in Q$ such that $q$ is in group $G_{i,0}$ for some $i$. Note that the poset formed by just considering the groups $G_{i,0}$ for $i \geq 0$ is exactly $P$ from Lemma~\ref{schur-positive 1-x}, which we know is upho. Then, note that $G_{i+k,l}$ is isomorphic to $G_{k,l}$ for $1 \leq l \leq k$ and $k \geq 1$, and the cover relations between $G_{i+k,l}$ and $G_{i+k+1,l+1}$ are the same as those between $G_{k,l}$ and those between $G_{k+1,l+1}$ as well. This means that $V_q$ is isomorphic to $Q$, as desired. 

Now, take $q \in Q$ such that $q$ is not in $G_{i,0}$ for some $i$. We strong induct on the rank of $Q$ that $q$ belongs to. The base case is true by construction (the elements in $G_{1,1}$ form $b-1$ disjoint posets that are all isomorphic to $Q$). For the inductive step, assume that for all $q \in G_{i,j}$ for $i \leq k$ and $0 < j \leq i$, $V_q$ is isomorphic to $Q$. By construction, if $q \in G_{i+1,1}$, $V_q$ is isomorphic to $Q$ for the same reason as the base case. 

If $q \in G_{i+1,k}$ for $k \geq 2$, then it is in a group of cardinality $(b-1)\cdot r_{k-1}$. Specifically, it is a member of one of the $b-1$ disjoint posets stemming from group $G_{i-k+2,1}$, which we know is isomorphic to $Q$ by the inductive hypothesis. Call this poset $Q'$. This means $q$ is in $G'_{k-1,l}$ for some $0 \leq l \leq k-1$, where $G'_{k-1,l}$ is the group in $Q'$ isomorphic to group $G_{k-1,l}$ in $Q$. Since $\rho(q') < \rho(q)$, $V_{q'}$ is isomorphic to $Q$ by strong induction, where $q'$ is the corresponding element to $q$ in $G_{k-1,l}$. Thus, $V_q$ is isomorphic to $Q'$, which is isomorphic to $Q$, as desired.
\end{proof}

Finally, we extend this construction to all upho posets with Schur-positive Ehrenborg quasisymmetric rank-generating functions with rational zeros and poles.

\begin{proof}[Proof of Theorem~\ref{schur-positive 1-bx}]
Note that there exists an upho poset with rank-generating function $\frac{1}{1-cx}$ for any $c$ (specifically, a $c$-tree). Furthermore, Lemma~\ref{schur-positive 1-bx mini} states that there exists an upho poset with rank-generating function $$\frac{(1+a_1x)(1+a_2x)\cdots(1+a_nx)}{1-bx}.$$

By Lemma~\ref{lem:multiplicative}, the product of two upho posets is upho and their rank-generating functions multiply, so we have that there exists an upho poset with rank-generating function $$\frac{(1+a_1x)(1+a_2x)\cdots(1+a_nx)}{(1-b_1x)(1-b_2x)\cdots(1-b_mx)}$$ for positive integers $a_1,a_2,\ldots,a_n,b_1,b_2,\ldots,b_m$, as desired.
\end{proof}

We propose the following conjecture, which is a more general case of Theorem \ref{schur-positive 1-bx}.

\begin{conjecture}\label{schur-positive conjecture}
A power series $f(t)\in 1+t\mathbb{Z}[[t]]$ is the rank-generating function of an upho poset $P$ whose Ehrenborg quasisymmetric function is a Schur-positive symmetric
function if and only if $f(t)$ has the form $g(t)/h(t)$, $g,h\in\mathbb{Z}[t]$, where the roots of $g(t)$ are real and negative, and the roots of $h(t)$ are real and positive.
\end{conjecture}

Note that Theorem~\ref{schur-positive 1-bx} is a special case of Conjecture~\ref{schur-positive conjecture} where the roots of $g(t)$ and $h(t)$ are rational.

\section{Planar upho posets}\label{sec: planar}

In this section, we classify the rank-generating functions of all planar upho posets. We begin with the following lemma. Although it is essentially known (see Corollary 2.4 of \cite{kelly1975planar}), we include the proof here for completeness. 

\begin{lemma}\label{lem:meet semilattice}
Every planar upho poset is a meet semilattice. In other words, any two vertices of a planar upho poset have a unique greatest lower bound.
\end{lemma}

\begin{proof}

Assume the contrary. Then two vertices of $P$, $v_1$ and $v_2$, do not have a meet. Because both vertices are greater than $\hat{0}$, this implies that $v_1$ and $v_2$ do not have a unique greatest lower bound. Suppose that $v_3$ and $v_4$ are incomparable and that they are both lower bounds of $v_1$ and $v_2$. Furthermore, suppose that any other lower bound of the two vertices is not greater than either $v_3$ of $v_4$. We consider all possible arrangements of the four elements, shown in Figure~\ref{fig: planar lemma}. For simplicity, chains are represented by straight lines.

In the leftmost image, the chains connecting $v_1$ and $v_2$ to $v_3$ and $v_4$ intersect. Because we assumed that there does not exist a lower bound greater than either $v_3$ or $v_4$, the chains cannot intersect at another vertex. However, the edges of the chains also cannot intersect due to our planarity condition, and the leftmost arrangement is invalid.

In the middle image, note that $v_3$ must cover a vertex $v_5$ on the same rank as $v_4$. If $v_4=v_5$, then $v_3>v_4$, which is a contradiction. Otherwise, the chain connecting $v_3$ to $v_5$ must intersect either the chain connecting $v_1$ to $v_4$ or $v_2$ to $v_4$. Because the poset is planar, the chains must intersect in another vertex, $v_6$. However, $v_3>v_6>v_4$, again a contradiction. Thus, the middle arrangement is invalid. 

Finally, in the rightmost image, there must exist a vertex $v_7$ on the same rank as $v_1$ that is greater than $v_2$. Then the chain connecting $v_2$ to $v_7$ must intersect with the chain connecting $v_1$ to $v_4$ at some vertex $v_8$. In this case, $v_1$ and $v_2$ are comparable with greatest lower bound $v_2$. Since $v_2$ is greater than both $v_3$ and $v_4$, we have achieved the desired contradiction. 
\end{proof}

\begin{figure}[h!]
 \centering
 \includegraphics[width = 370pt]{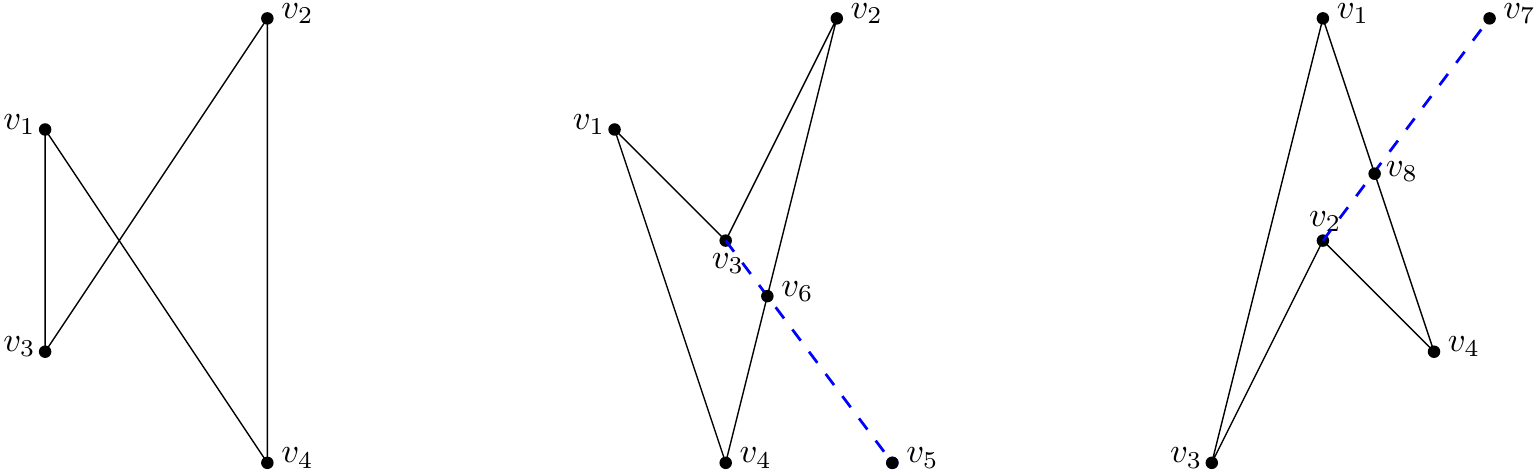}
 \caption{Example for Lemma~\ref{lem: b-1 bifurcated vertices} and Lemma~\ref{lem:meet semilattice}. Chains are represented by straight lines for simplicity.}
 \label{fig: planar lemma}
\end{figure}

\begin{lemma}\label{lem: cover 2 vertices}
In a planar upho poset, no vertex can cover more than $2$ other vertices.
\end{lemma}

\begin{proof}
Suppose a vertex $v$ covers $3$ vertices $v_1$, $v_2$, and $v_3$. Without loss of generality, let $v_2$ be between the other two vertices in a planar Hasse diagram of the poset. Then, no other vertices may cover $v_2$, contradicting the upho conditions.
\end{proof}

For clarity, we fix an arbitrary planar Hasse diagram for every planar upho poset. We call two vertices of a planar upho poset $P$ as ``adjacent" if they belong to the same rank and if, within the planar Hasse diagram of $P$, no other vertex lies between them. We can split all vertices that cover more than $1$ other vertex into the following two categories.

\begin{definition}\label{def: root-bifurcated}
A vertex $v$ of a planar upho poset is called \textit{root-bifurcated} if it covers exactly $2$ vertices which satisfy the following conditions:
\begin{enumerate}
    \item their meet is $\hat{0}$, and
    \item they are adjacent in the planar Hasse diagram of the poset.
\end{enumerate}
\end{definition}

\begin{definition}\label{def: bifurcated}
A vertex $v$ of a planar upho poset is called \textit{bifurcated} if it covers exactly $2$ vertices which satisfy the following conditions:
\begin{enumerate}
    \item their meet is not $\hat{0}$, and
    \item they are adjacent in the planar Hasse diagram of the poset.
\end{enumerate}
\end{definition}

It is important to note the following:

\begin{lemma}\label{lem: bifurcated}
In a planar upho poset, any vertex $v$ covering multiple other vertices must be either \textit{bifurcated} or \textit{root-bifurcated}. 
\end{lemma}

\begin{proof}
By Lemma~\ref{lem: cover 2 vertices}, $v$ must cover exactly $2$ vertices $v_1$ and $v_2$. Furthermore, by Lemma~\ref{lem:meet semilattice}, $v_1$ and $v_2$ must have a meet (satisfying condition $1$ of either \textit{bifurcated} or \textit{root-bifurcated} vertices). Thus, we are left to that $v_1$ and $v_2$ are adjacent in the planar representation of the poset (condition $2$). However, as in the proof of Lemma~\ref{lem: cover 2 vertices}, any vertex between $v_1$ and $v_2$ may not be covered by any vertices, which contradicts the upho conditions. Thus, $v_1$ and $v_2$ must be adjacent in the planar Hasse diagram of the poset, as desired.
\end{proof}

From Definition~\ref{def: root-bifurcated}, we obtain the following lemma.

\begin{lemma}\label{lem: adj atoms}
Any two vertices $v_1$ and $v_2$ that are covered by a \textit{root-bifurcated} vertex $v$ must each be greater than exactly one atom, and those atoms must be adjacent in the planar Hasse diagram of the poset.
\end{lemma}

\begin{proof}
Let $v_1$ and $v_2$ be on rank $k$ of some poset $P$ with up-degree $b$. Note that the meet of $v_1$ and $v_2$ must be $\hat{0}$, so the two vertices cannot cover the same atom. Furthermore, by definition, $v_1$ and $v_2$ must be adjacent in the planar representation of the poset. Suppose by contradiction that $v_1$ and $v_2$ are greater than non-adjacent vertices. Then consider the chain $c$ beginning any atom between the two and consisting of a series of vertices that are neither the rightmost nor the leftmost covering the previous vertex. Note that this is possible because in order for such an atom to exist, the poset has to have up-degree of at least $3$. None of these vertices can be \textit{bifurcated} without contradicting the planarity conditions, as demonstrated by the vertex labelled $x$ in Figure~\ref{fig: adj atoms diagram}. Thus, $c$ cannot intersect either of the chains connecting $v_1$ and $v_2$ to $\hat{0}$. Because the two chains begin on opposite sides of $c$, the vertex belonging to rank $k$ of this chain must be between $v_1$ and $v_2$, a contradiction. Thus, $v_1$ and $v_2$ must cover adjacent atoms.

We now prove that $v_1$ and $v_2$ must each be greater than exactly one atom. Without loss of generality, let $v_1$ be greater than $y_1$ and $y_2$ and let $v_2$ be greater than $y_3$ such that $y_1$ is to the left of $y_2$ and $y_2$ is to the left of $y_3$. Consider a chain $c'$ with the same conditions as before beginning with $y_2$. Again, the vertices of this chain cannot be bifurcated because this would contradict the planarity conditions. However, the chains connecting $v_1$ to $y_1$ and $v_2$ to $y_3$ are on opposite sides of $c'$. Thus, the vertex on rank $k$ of $c$ must lie between $v_1$ and $v_2$, and we obtain a contradiction.
\end{proof}

\begin{figure}[h!]
 \centering
 \includegraphics[width = 100pt]{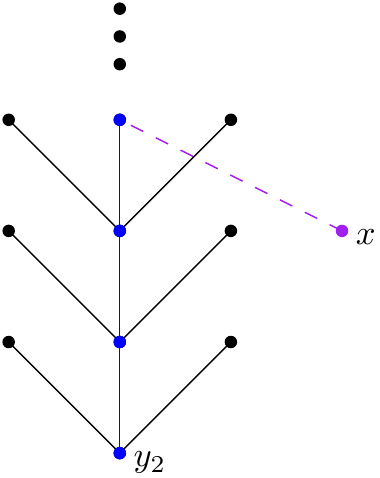}
 \caption{Diagram for Lemma~\ref{lem: adj atoms}. The bottom segment of chain $c$ is shown in blue.}
 \label{fig: adj atoms diagram}
\end{figure}

This implies the following key lemma.

\begin{lemma}\label{lem: b-1 bifurcated vertices}
There are at most $b-1$ \textit{root-bifurcated} vertices in a planar upho poset of up-degree $b$.
\end{lemma}

\begin{proof}
By Lemma~\ref{lem: adj atoms}, any \textit{root-bifurcated} vertex must cover two vertices each greater than exactly one atom, where those atoms are adjacent. We claim that we cannot have two \textit{root-bifurcated} vertices $v_1$ and $v_2$ that are greater than the same pair of atoms, named $v_3$ and $v_4$. Because $v_1$ and $v_2$ lie on the same rank, and $v_3$ and $v_4$ lie on the same rank, we actually arrive at a special case of the leftmost diagram of Figure~\ref{fig: planar lemma}. Due to planarity, the chains must intersect at a vertex, but this implies that some of the vertices covered by $v_1$ and $v_2$ are no longer each greater than exactly one atom. Thus, we have arrived at a contradiction, so it is not possible for two \textit{root-bifurcated} vertices to both be greater than the same pair of atoms. Then, because there are $b-1$ pairs of adjacent atoms in the planar Hasse diagram, we can only have $b-1$ \textit{root-bifurcated} vertices.
\end{proof}

Denote by $a_i$ the number of \textit{root-bifurcated} vertices on rank $i$ where $\sum a_i \leq b-1$. Finally, we are ready to state our theorem. 

\begin{theorem}\label{thm:planar-rgf}
The rank-generating function of any planar poset $P$ with up-degree $b$ is of the form $Q(x)^{-1}$ where $Q(x)=1-bx+a_2x^2+a_3x^3+ \cdots +a_nx^n$ such that $b, a_1, a_2, \ldots, a_n \in \mathbb{Z}_{\geq{0}}$ where the coefficients sum to at most $0$, and $a_i$ refers to the number of \textit{root-bifurcated} vertices on rank $i$.
\end{theorem}

\begin{proof} 
Let $P$ be a planar upho poset with up-degree $b$. We count the number of vertices on rank $i$, denoted $r_i$, as follows. 

We begin with an upper bound of $b\cdot r_{i-1}$ vertices because each vertex on rank $i-1$ has up-degree $b$. This is precisely the number of edges between rank $i-1$ and rank $i$, but it overcounts the vertices on rank $i$ which cover multiple vertices on rank $i-1$. By Lemma~\ref{lem: bifurcated}, all such vertices cover exactly $2$ other vertices and are either \textit{bifurcated} or \textit{root-bifurcated}. Thus, we can calculate the number of vertices on rank $i$ by subtracting the number of \textit{bifurcated} and \textit{root-bifurcated} vertices on rank $i$ from $b\cdot r_{i-1}$.

Any \textit{bifurcated} vertex $v$ on rank $i$ covers two vertices on rank $i-1$ with meet on rank $i-k$ for some $k<i$. In fact, given $k$, we can precisely count the number of such vertices $v$. Consider the poset $V_w$ above any vertex $w$ on rank $i-k$. Then because of the upho conditions, there must be exactly $a_k$ \textit{root-bifurcated} vertices on rank $k$ of the subposet $V_w$. This implies that there are $a_k \cdot r_{i-k}$ \textit{bifurcated} vertices $v$ on rank $i$ of the original poset such that the meet of the vertices covered by $v$ is on rank $i-k$. Note that we are not overcounting because every pair of vertices has a \textit{unique} meet. Lastly, by definition, if we let $k=i$ then the number of \textit{root-bifurcated} vertices on rank $i$ is $a_i$.

We can sum over all $k$ to obtain the total number of vertices on rank $i$ that cover multiple other vertices: $\sum_{k=1}^i r_{i-k} \cdot a_k$. We subtract this value from $b \cdot r_{i-1}$ and obtain $r_i = b \cdot r_{i-1} - \sum_{k=1}^i r_{i-k} \cdot a_k$.

To prove that the closed form of our rank-generating function is $$\frac{1}{1-bx+a_2x^2+a_3x^3+a_4x^4+\cdots},$$ we show that
$$(1-bx+a_2x^2+a_3x^3+\cdots)(r_0+r_1x+r_2x^2+r_3x^3 + \cdots ) = 1.$$

When the left side of the equation is expanded, the constant term is trivially $1$. Then the coefficient of any $x^i$ for $i>0$ is $r_i - b\cdot r_{i-1} + \sum_{k=1}^i a_{i-k} \cdot r_k = 0$, as desired. Note that the coefficients of the denominator of the rank-generating function sum to at most $0$ as a consequence of Lemma~\ref{lem: b-1 bifurcated vertices}, and our proof is complete.
\end{proof}

In fact, we can extend the ideas from this proof to show that there exists a planar upho poset for any rank-generating function of this form. 

\begin{theorem}\label{thm:planar-construction}
Given any $Q(x)=1-bx+a_2x^2+a_3x^3+ \cdots +a_nx^n$ such that $b, a_1, a_2, \ldots, a_n \in \mathbb{Z}_{\geq{0}}$ and the coefficients sum to at most $0$, there exists at least one planar upho poset $P$ with rank-generating function $Q(x)^{-1}$.
\end{theorem}

\begin{proof}
We construct an upho poset of up-degree $b$ with $a_i$ \textit{root-bifurcated} vertices on rank $i$. As shown in the proof of Theorem~\ref{thm:planar-rgf}, such a poset has the desired rank-generating function.

We begin with one vertex on rank $0$ and $b$ vertices on rank $1$ and construct upwards.
When we have constructed the poset up to rank $i-1$, we draw $b$ vertices on rank $i$ covering each vertex on rank $i-1$. We  combine pairs of adjacent vertices $v_1$ and $v_2$ on rank $i$ covering $s_1$ and $s_2$ on rank $i-1$ respectively. We do this by replacing $v_1$ and $v_2$ by a single vertex $v_3$, placed between the two original vertices in the planar representation, that covers both $s_1$ and $s_2$. See steps $2$ and $3$ of Figure~\ref{fig: planar} for an example. Note that $v_3$ is either a \textit{bifurcated} or \textit{root-bifurcated}. In order to choose which vertices to combine in order to satisfy both the desired rank-generating function and the upho conditions, we perform the following process.

Consider all subposets $V_v$ above some vertex $v$ on rank $0<j<i$ of the poset. Within rank $i-j$ of this subposet, we add in all of the vertices isomorphic to the \textit{root-bifurcated} vertices on rank $i-j$ of $P$. Specifically, if there exists a \textit{root-bifurcated} vertex greater than the $x$ and $(x+1)^{\text{st}}$ atoms on rank $1$ (counting from the left), then we combine adjacent vertices above the $x$ and $(x+1)^{\text{st}}$ atoms of the subposet. We perform this process for all such subposets $V_v$. Finally, we combine vertices to form $a_i$ \textit{root-bifurcated} vertices wherever possible. An example for the rank-generating function $(1-3x+x^2+x^3)^{-1}$ can be seen in Figure~\ref{fig: planar}.

To prove that this poset is upho, consider the subposet $V_w$ above some vertex $w$ on rank $k$ of the poset. Notice that the first two ranks of $V_w$ are trivially isomorphic to the first two ranks of $P$. Furthermore, when we construct rank $i$ of $P$ (for $i>w$), we consider every subposet $V_z$ for $z$ in $V_w$ and combine vertices as described above. Notice that this process is identical to the process we performed when constructing rank $i-w$ of $P$. By induction, this implies that $V_w \cong P$. Thus, the poset is upho, as desired.
\end{proof}

\begin{figure}
 \centering
 \includegraphics[width = 430pt, height=540pt]{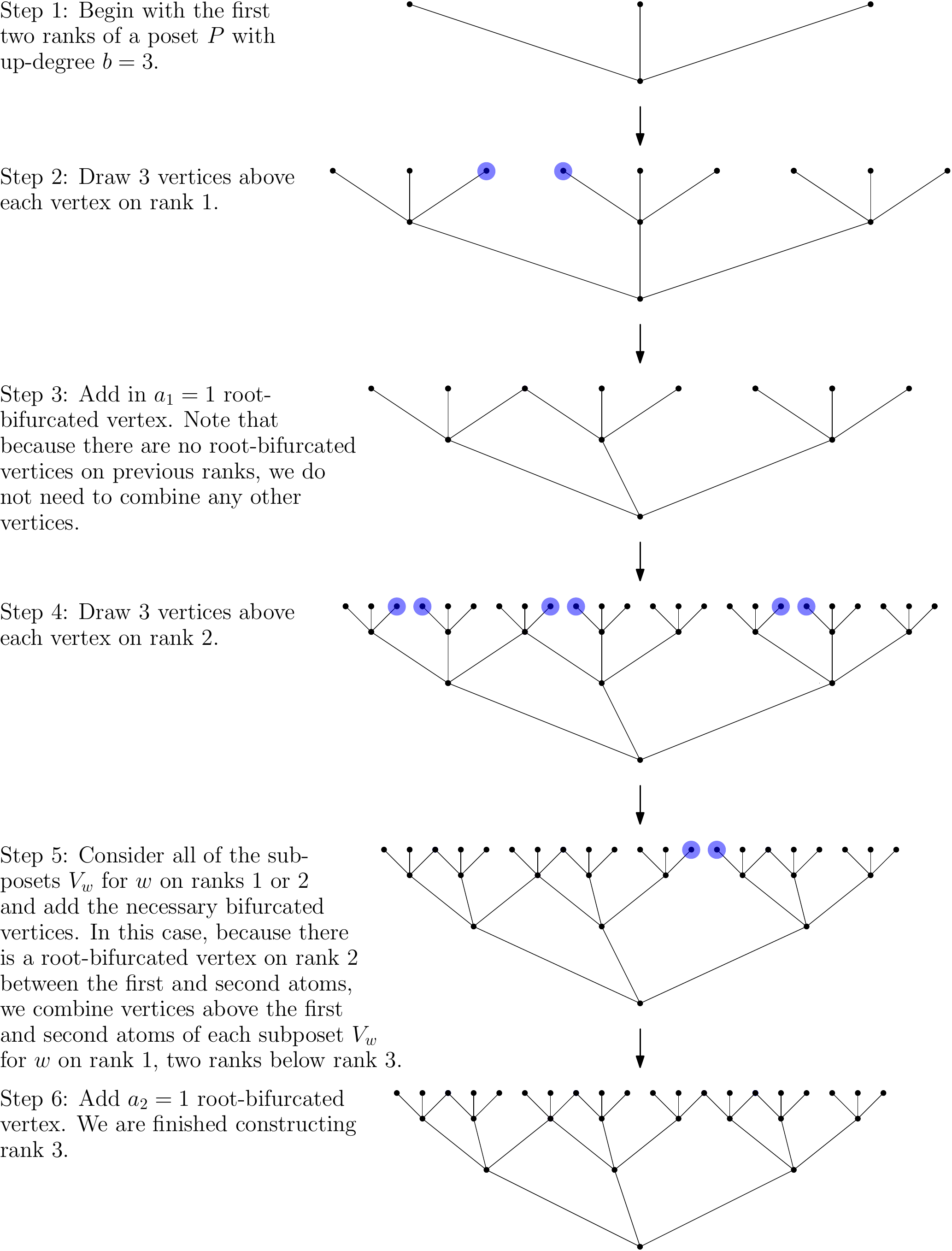}
 \caption{Example of the construction for Theorem~\ref{thm:planar-construction} for rank-generating function $(1-3x+x^2+x^3)^{-1}$. Vertices that are combined in the following step are highlighted in blue.}
 \label{fig: planar}
\end{figure}

\section{Upho posets with uncomputable rank-generating function}\label{sec: irrationality}

So far, we have focused on upho posets with rational rank-generating functions. Thus, a natural question arises: does there exist an upho poset with an uncomputable rank-generating function?

We take a different approach than we have in the previous sections, in that rather than directly constructing a poset satisfying the desired criteria, we prove the existence of an uncountable set of upho posets with distinct rank-generating function, and show that at least one poset in the set must have uncomputable rank-generating function. Furthermore, we no longer depend on the Hasse diagram representation of the poset. Instead, we construct our posets from strings of letters.

Consider a vocabulary of symbols $\Sigma$ and a possibly infinite set of homogeneous relations of the form $b_1b_2\ldots b_k = c_1c_2\ldots c_k$, where $b_i,c_i \in \Sigma$ for all $1 \leq i \leq k$. We define a monoid $M$ consisting of every finite string of letters from the alphabet with the set of relations under the binary operation of concatenation. It is clear that concatenation is associative and the identity element is the empty string. We then let $P(M)$ be the associated poset consisting of elements of the monoid, where $X$ covers $Y$ if $X=Ya_i$ for some $a_i \in \Sigma$ and $AB$ represents concatenation of the strings $A$ and $B$.

\begin{example} Suppose we define a monoid with alphabet $a$, $b$, and $c$ with relations $ac = ba$ and $bc = ca$. In this case, the vertex represented by both $ac$ and $ba$ covers the element represented by $a$ and the element represented by $b$. Likewise, the vertex represented by $bc$ and $ca$ covers both the vertex represented by $b$ and the vertex represented by $a$.
\end{example}

We define some strings to be equal as follows. Firstly, if two strings are equivalent in the monoid, they represent the same element in our poset. We do not allow equalities involving strings of length $1$ (as this is equivalent to decreasing the size of the alphabet). Furthermore, if $A$, $B$, $X$, and $Y$ are strings, then $A = B$ implies $XAY = XBY$. That is, the strings $XAY$ and $XBY$ are associated with the same vertex in the associated poset. For the remainder of this section, we will refer to the relationship between $XAY$ and $XBY$ as a \textit{string substitution}.

Note that as in Section~\ref{sec: planar}, we are describing upho posets using the vertices that cover multiple others. Furthermore, the set of equalities defining a poset will uniquely determine its rank-generating function.

\begin{remark}
This is not a method to construct \textit{all} upho posets.
\end{remark}

We provide some examples of posets defined as described.

\begin{example}
A full binary tree, shown in Figure~\ref{fig: future work} is uniquely described by an alphabet of size $2$ and no relations.
\end{example}
\begin{example}
The Stern Poset, also shown in Figure~\ref{fig: future work}, is uniquely described by an alphabet of size $3$ and the relations $ac = ba$ and $bc = ca$. 
\end{example}

\begin{figure}[h!]
 \centering
 \includegraphics[width = 470pt]{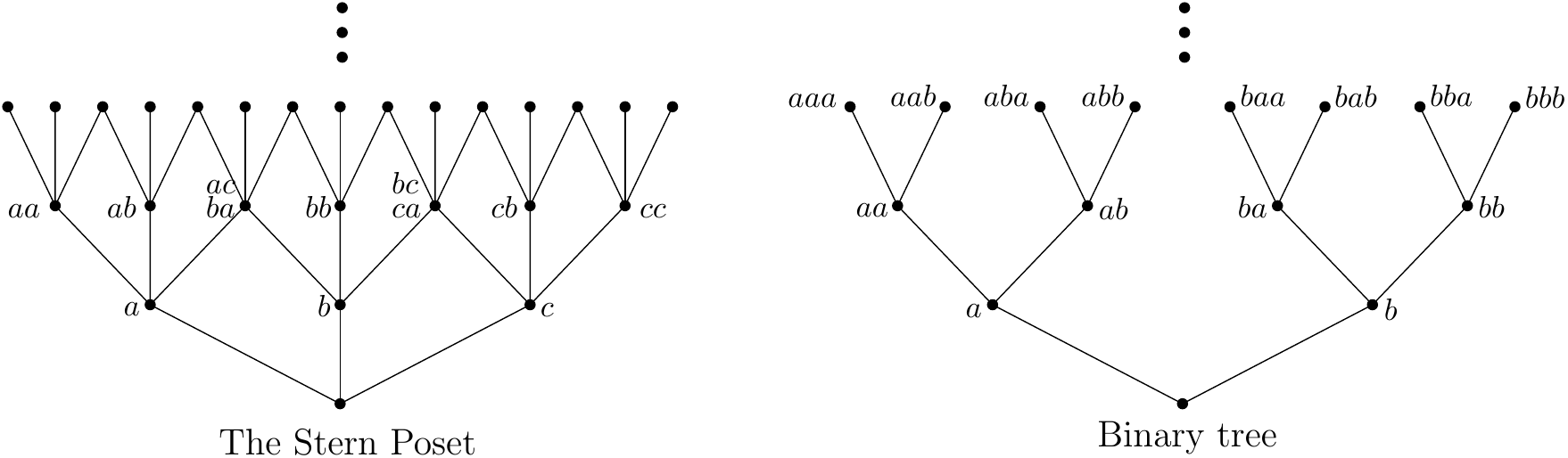}
 \caption{Example of labelled vertices. Due to space constraints, the labels on rank $3$ of the Stern Poset have been omitted, so from left to right they are: $aaa$, $aab$, $aba = aac$, $abb$, $baa = aca = abc$, $bab = acb$, $bba = bac = acc$, $bbb$, $caa = bca = bbc$, $cab = bcb$, $cba = cac = bcc$, $cbb$, $cca = cbc$, $ccb$, and $ccc$.}
 \label{fig: future work}
\end{figure}

We wish to reduce the upho conditions for posets described by such relations. In particular, we note that a poset is upho if $AX = AY$ implies $X = Y$. Intuitively, if the string $A$ is associated with some vertex $v$, then removing $A$ from all vertices greater than or equal to $v$ creates a map $\iota: V_v \rightarrow P$. Then, $AX = AY$ implies $X = Y$, the relations are preserved under $\iota$, so we claim that $\iota$ is an isomorphism. (Note that this condition is not equal to, but implies the upho conditions). We state this formally as follows.

\begin{lemma}
Given a monoid $M$, the associated poset $P(M)$ is upho if $AX = AY$ implies $X = Y$ for all $A,X,Y \in M$.
\end{lemma}

\begin{proof}
We wish to show that, if $P$ satisfies the given criteria, then for any string $B \in P$, $V_B \equiv P$. Note that any element of $V_B$ begins with $B$, so it is sufficient to show that the map $\iota: BX \rightarrow X$ is an isomorphism. To do this, first note that $BX \leq BY$ if and only if $X \leq Y$, which implies that $\iota$ is a homomorphism. Also, $\iota$ is clearly a bijection, so $\iota$ is an isomorphism and we are done.
\end{proof}

In order to prove the existence of a poset with uncomputable rank-generating function, we begin by defining the following class of infinite posets. Consider the alphabet $\{R, L\}$ with the following set $S$ of relations.

\begin{align*}
    LRLRLL &= RRLLRL \\
    LRLRLRLL &= RRLLLLRL \\
    LRLRLRLRLL &= RRLLLLLLRL \\
       & \vdotswithin{ = }\notag \\
\end{align*}

In particular, for every $n \geq 2$, $S$ contains a relation between a string consisting of $n$ copies of $LR$ followed by $LL$ and a string consisting of $RR$ followed by a string of $2 \cdot (n-1)$ copies of $L$, and $RL$. For each $n \geq 2$, call the associated relation $t_n$. This means that $S=\{t_2,t_3,\ldots\}$. Note that the left and right sides of $t_n$ will always be of the same length, namely $2n+2$.

We note the following.

\begin{lemma} \label{irrational-upho}
Every poset $P$ defined using the alphabet $\{L, R\}$ and the relations in an element of $\mathcal{P}(S)$ is upho.
\end{lemma}
\begin{proof}

We begin by showing that, for strings $X$ and $Y$, $LX = LY$ or $RX = RY$ implies $X = Y$. We will prove later that this is sufficient.

\textbf{Case 1.} Suppose that $LX=LY$. By definition, there exists a sequence of string substitutions that changes $LX$ into $LY$. We assume for simplicity that if we remove two substitutions from our sequence, the resulting sequence of substitutions does not change $LX$ to $LY$.

Note that if none of these string substitutions include the first $L$ of $LX$, then the same sequence can be used to change $X$ into $Y$. In this case, it is clear that $X = Y$, as desired.

Thus, we assume that the first $L$ in $LX$ is modified in a sequence of moves transforming $LX$ to $LY$. Note that at some point during our sequence of substitutions, $LX$ must begin with $LRLRLR\ldots LRLL$, so that we may substitute the first $L$. Say this is the left side of $t_{n}$ for some $n$. Performing the appropriate move yields $$\underbrace{LRLRLR \ldots LR}_\text{$n$ $LR$s}LL \ldots = RR\mathrlap{\underbrace{\phantom{LL\ldots LL}}_{\text{$2n-2$ $L$s}}}LL\ldots LLRL \ldots.$$ 

We must eventually change the first $R$ to an $L$ in order to obtain $LY$. Thus, we must eventually apply a substitution that changes the first part of the string, $RRLL \ldots LLRL$. Say we do this with relation $t_m$. There are two possibilities for this move:

\begin{enumerate}
    \item       $RR\mathrlap{\underbrace{\phantom{LL\ldots LL}}_{\text{$2n-2$ $L$s}}} LL\ldots L \mathrlap{\overbrace{\phantom{LRLRL \ldots LR}}^{\text{$m$ $LR$s}}}LRLRL \ldots LRLL \ldots = RR\mathrlap{\underbrace{\phantom{LL \ldots LL}}_{\text{$2n-3$ $L$s}}}LL \ldots LL RR\mathrlap{\underbrace{\phantom{LL \ldots LL}}_{\text{$2m-2$ $L$s}}}LL \ldots LLRL \ldots$
    \item $RR\mathrlap{\underbrace{\phantom{LL\ldots LL}}_{\text{$2n-2$ $L$s}}} LL\ldots LLR \mathrlap{\underbrace{\phantom{LRLRL \ldots LR}}_{\text{$m$ $LR$s}}}LRL \ldots LRLL \ldots = RR\mathrlap{\underbrace{\phantom{LL \ldots LL}}_{\text{$2n-2$ $L$s}}}LL \ldots LL RRR\mathrlap{\underbrace{\phantom{LL \ldots LL}}_{\text{$2m-2$ $L$s}}}LL \ldots LLRL \ldots$
\end{enumerate}

Note, in the former instance, that the resulting string has $2n-1$ $L$s (an odd number) in between the first $RR$, and the following $R$. However, the number of consecutive $L$s after the first $RR$ in every $t_i$ is even. Furthermore, in the second instance, the resulting string has three consecutive $R$s in the middle. Thus, we cannot apply any substitutions involving the first segment of these strings, and our subsequent substitutions must be applied to the segment of each string beginning with $RRLL\ldots LRL$. However, this is precisely the situation that we were in before performing a substitution. Thus, because every substitution we perform will consist of letters to the right of the previous substitution, we will never change the first $R$ back to an $L$. Hence, if $LX = LY$, we must have $X = Y$, as desired.

We consider the second case.

\textbf{Case 2.} Let $RX$=$RY$. As in the first case, we attempt a sequence of substitutions that change $RX$ into $RY$. Again, we must eventually perform a move changing the first $R$ to an $L$. Thus, we have $$RR\mathrlap{\underbrace{\phantom{LL \ldots LL}}_\text{$2n-2$ $L$s}}LL \ldots LL RL \ldots = \mathrlap{\underbrace{\phantom{LRLR \ldots LR}}_{\text{$n$ $LR$s}}}LRLR \ldots LR LL \ldots.$$ In order to change the first $L$ back to an $R$ without directly undoing our first substitution, we must eventually perform a substitution that modifies some of the letters that we changed in our first substitution. We have only one way to do this: 
$$\mathrlap{\underbrace{\phantom{LRLR \ldots LR}}_{\text{$n$ $LR$s}}}LRLR \ldots LR L\mathrlap{\underbrace{\phantom{LRLR \ldots LR}}_\text{$m$ $LR$s}}LRLR \ldots LR LL \ldots  = \mathrlap{\underbrace{\phantom{LRLR \ldots LR}}_{\text{$n$ $LR$s}}}LRLR \ldots LR R\mathrlap{\underbrace{\phantom{LL \ldots LL}}_\text{$2m-2$ $L$s}}LL \ldots LL RL \ldots.$$ 
However, once again, no substitutions can  be performed to the beginning of the string, $LRLRLR \ldots LRR$. So, this is equivalent to the case in which our string began with $RRLL \ldots LLRL$. Thus, similarly to Case $1$, $RX = RY$ implies $X = Y$, as desired.

Thus, we have proved that $LX = LY$ or $RX = RY$ implies $X = Y$. Now let $A = a_1 a_2 a_3  \ldots  a_n$ where $a_i \in \{R, L\} \text{ } \forall \text{ } 0<i\leq n$, and consider the equality $$a_1 a_2 a_3  \ldots  a_n X = a_1 a_2 a_3  \ldots  a_n Y.$$ Note that, for all $1 \leq i \leq n$, $$a_i a_{i+1} a_{i+2} \ldots a_n X = a_i a_{i+1} a_{i+2} \ldots a_n Y \implies a_{i+1} a_{i+2} a_{i+3} \ldots a_n X = a_{i+1} a_{i+2} a_{i+3} \ldots a_n Y.$$ This eventually implies $X=Y$, as desired.
\end{proof}

In order to prove that at least one of the posets defined by the alphabet $\{R, L\}$ and the relations in an element of $\mathcal{P}(S)$ (the power set of $S$, where $S$ is the set of relations defined above) has uncomputable rank-generating function, we consider the following lemma.

\begin{lemma} \label{irrationality distinct rgf}
Any two posets defined by different elements of $\mathcal{P}(S)$ have different rank-generating functions. 
\end{lemma}

\begin{proof}
We first show that $t_n$ is not implied by any of the previous for all $n \geq 2$. Notice that there is no substring $LR$ of $RRLL \ldots LLRL$, so any string substitution we perform must begin with $RR$.

However, the only substring starting with $RR$ that we can perform would be to substitute the entire string. Thus, none of the elements of $S$ can be replaced by a sequence of smaller string substitutions, as desired.

Now, consider two posets $P$ and $Q$ defined by distinct elements of $\mathcal{P}(S)$, called $p$ and $q$ respectively. We consider the shortest relation in $S$ that is contained in exactly one of $p$ and $q$, and let the length of the relation be $l$. Then, the number of elements on rank $l$ of $P$ and $Q$ differs by exactly one, leading to the two having different rank-generating functions, as desired.
\end{proof}

Finally, we are ready to prove Theorem~\ref{thm:intro-irrationality}.

\begin{proof}[Proof of Theorem~\ref{thm:intro-irrationality}]

Note that $\mathbb{Z}[x]$ is countably infinite, which implies that $\mathbb{Z}[x] \times \mathbb{Z}[x]$ is countably infinite. The latter quantity is equivalent to the number of rational functions where both the numerator and denominator have integer coefficients. However, $|\mathcal{P}(S)|$ is uncountable for an infinite set $S$. By Lemma \ref{irrational-upho}, every element of $\mathcal{P}(S)$ corresponds to an upho poset, and by Lemma \ref{irrationality distinct rgf}, no two have the same rank-generating function. Thus, at least one element of $\mathcal{P}(S)$ must have an uncomputable rank-generating function, as desired.
\end{proof}

\section{Acknowledgments}

We would like to thank Richard Stanley for proposing this problem and for insightful conversations. We would also like to thank the MIT PRIMES program, under which this research was conducted.

\bibliographystyle{plain}

\end{document}